\documentclass[a4paper,fleqn]{cas-dc}

\usepackage[numbers]{natbib}
\usepackage{tikz}
\usetikzlibrary{matrix,patterns,positioning,decorations,arrows,shapes,decorations.markings,calc,fit}
\usepackage{adjustbox}
\usepackage{geometry}
\usepackage{todonotes}
\usepackage{amsmath}
\usepackage{amsthm}
\usepackage{bm}
\usepackage{pgfplots}
\pgfplotsset{compat=1.18}
\usepackage{subfiles}
\usepackage{subcaption}

\newcommand{\CitetVR}[1]{Van Rossum et al.\ \cite{#1}}

\newtheorem{thm}{Theorem}

\newtheorem{prop}[thm]{Proposition}

\begin{document}

\let\WriteBookmarks\relax
\def\floatpagepagefraction{1}
\def\textpagefraction{.001}
\shorttitle{Enforcing TSP-Optimality in Fair Vehicle Routing}
\shortauthors{Bart van Rossum, Rui Chen and Andrea Lodi}

\title [mode = title]{Enforcing TSP-Optimality in Fair Vehicle Routing by Cutting Planes}
\author[1]{Bart {van Rossum}}[orcid=0000-0002-8234-5373]
\address[1]{Operations, Planning, Accounting \& Control, Eindhoven University of Technology, Eindhoven, The Netherlands}
\corref{cor1}
\author[2]{Rui Chen}[orcid=0000-0002-8848-6118]
\address[2]{School of Data Science, The Chinese University of Hong Kong, Shenzhen, Guangdong, China}
\author[3]{Andrea Lodi}[orcid=0000-0001-9269-633X]
\address[3]{Cornell Tech, New York City, NY, USA}

\cortext[cor1]{Corresponding author: b.t.c.v.rossum@tue.nl}

\begin{abstract}
\noindent 
We study the fair capacitated vehicle routing problem, in which a fleet of vehicles must serve a set of customers such that the difference between the longest and shortest route, the range, is minimized. A key challenge is that the range objective is non-monotonic: it can be reduced by artificially lengthening routes, leading to solutions that violate TSP-optimality of individual routes. Existing exact methods struggle to handle this efficiently. We propose a branch-price-and-cut framework that enforces TSP-optimality through TSP-optimality cuts, which forbid TSP-dominated arc sequences. We strengthen the cuts through a dedicated lifting procedure. Computational experiments on benchmark instances with up to 25 customers show the method solves nearly all instances to optimality, achieving an average gap of 0.27\% on the hardest configurations.
\end{abstract}


\maketitle

\section{Introduction}
\label{sec:intro}

The goal of vehicle routing is to assign customers to vehicle routes so as to minimize total routing distance. Minimizing distance, however, may lead to highly imbalanced workloads across drivers, which can be undesirable from a practical perspective. This has motivated a growing interest in \emph{fair vehicle routing}, where the goal is to find a set of routes that is balanced according to some fairness measure \citep{agius2026solution, bruinink2025beyond, ljubic2026fairvehicle, matl2018workload, van2024new, rossum2025efficient}. A natural and widely used measure of imbalance is the \emph{range}, defined as the difference between the longest and shortest route. Minimizing the range promotes an equitable distribution of workload among drivers, and serves as the fairness objective throughout this paper.

A particular challenge that arises when minimizing range over route distances is that the objective is \emph{non-monotonic} \citep{matl2018workload}: the range can sometimes be reduced by artificially lengthening a short route, for example by visiting customers in a suboptimal order. As a result, an optimal solution may include routes that are not the shortest tours over their visited customers, i.e., routes that violate \emph{TSP-optimality}. This issue is absent when minimizing routing distance, where the monotonicity of the objective guarantees that TSP-optimal routes are always preferred. Figure~\ref{fig:tsp_suboptimality} illustrates this phenomenon. While TSP-violating solutions are optimal for the range objective, they are practically undesirable: a driver following a TSP-violating route incurs unnecessary travel time. Enforcing TSP-optimality, i.e., requiring that each route visits its customers in a shortest-possible order, is therefore a natural and practically motivated constraint. 

Throughout this paper, we distinguish between the \emph{fair capacitated vehicle routing problem} (F-CVRP), which minimizes the range without restricting route structure, and the \emph{fair capacitated vehicle routing problem with TSP-optimality} (F-CVRP-TSP), which additionally requires each route to be a shortest tour over its visited customers.

\begin{figure*}[htb!]
\centering
\begin{subfigure}[t]{0.48\textwidth}
\centering
\begin{tikzpicture}[scale=1,
    midarrow/.style={
        decoration={markings, mark=at position 0.5 with {\arrow{stealth}}},
        postaction={decorate}
    }
]
\begin{axis}[
    axis lines = none,
    width = \textwidth,
    height = \textwidth,
    disabledatascaling,
    xmin=-50, xmax=950, ymin=-50, ymax=950,
    ]
\tikzset{every node/.style={draw=black, fill=black!15, shape=circle, scale=0.5}}
    \node[shape=rectangle, scale=2, fill=white] at (500, 80) (0) { };
    \node at (600, 400) (1) { };
    \node at (750, 600) (2) { };
    \node at (850, 350) (3) { };
    \node at (350, 250) (4) { };
    \node at (180, 370) (5) { };
    \node at (120, 570) (6) { };
    \node at (280, 710) (7) { };
    \draw[black, midarrow] (0) -- (1);
    \draw[black, midarrow] (1) -- (2);
    \draw[black, midarrow] (2) -- (3);
    \draw[black, midarrow] (3) -- (0);
    \draw[black, midarrow] (0) -- (4);
    \draw[black, midarrow] (4) -- (5);
    \draw[black, midarrow] (5) -- (6);
    \draw[black, midarrow] (6) -- (7);
    \draw[black, midarrow] (7) -- (0);
\end{axis}
\end{tikzpicture}
\caption{TSP-optimal solution with range 227.}
\end{subfigure}
\hspace{0.25cm}
\begin{subfigure}[t]{0.48\textwidth}
\centering
\begin{tikzpicture}[scale=1,
    midarrow/.style={
        decoration={markings, mark=at position 0.5 with {\arrow{stealth}}},
        postaction={decorate}
    }
]
\begin{axis}[
    axis lines = none,
    width = \textwidth,
    height = \textwidth,
    disabledatascaling,
    xmin=-50, xmax=950, ymin=-50, ymax=950,
    ]
\tikzset{every node/.style={draw=black, fill=black!15, shape=circle, scale=0.5}}
    \node[shape=rectangle, scale=2, fill=white] at (500, 80) (0) { };
    \node at (600, 400) (1) { };
    \node at (750, 600) (2) { };
    \node at (850, 350) (3) { };
    \node at (350, 250) (4) { };
    \node at (180, 370) (5) { };
    \node at (120, 570) (6) { };
    \node at (280, 710) (7) { };
    \draw[black, midarrow] (0) -- (1);
    \draw[black, midarrow] (1) -- (3);
    \draw[black, midarrow] (3) -- (2);
    \draw[black, midarrow] (2) -- (0);
    \draw[black, midarrow] (0) -- (4);
    \draw[black, midarrow] (4) -- (5);
    \draw[black, midarrow] (5) -- (6);
    \draw[black, midarrow] (6) -- (7);
    \draw[black, midarrow] (7) -- (0);
\end{axis}
\end{tikzpicture}
\caption{TSP-violating solution with range 87.}
\end{subfigure}
\caption{Instance with two vehicles where TSP-suboptimality reduces the range.}
\label{fig:tsp_suboptimality}
\end{figure*}
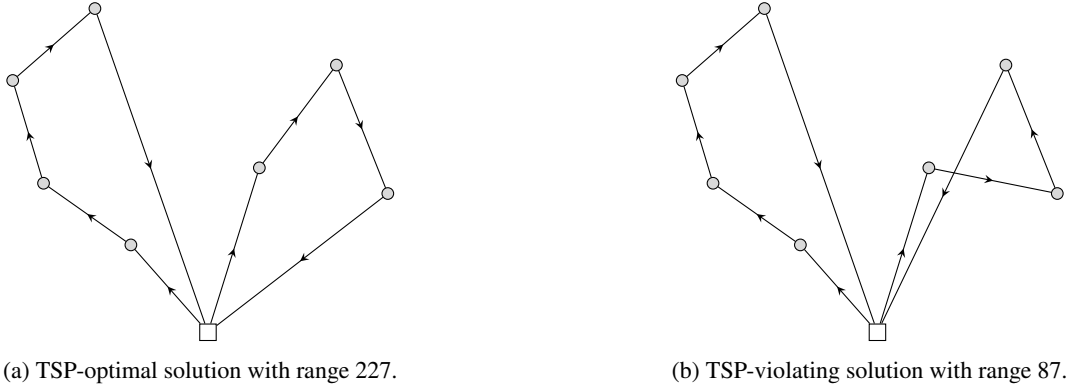

Exact methods for F-CVRP have been developed only recently. \CitetVR{rossum2025efficient} propose a branch-and-price algorithm for minimizing range that solves instances with up to $n = 25$ customers, but does not enforce TSP-optimality. They extend their framework to F-CVRP-TSP in two ways. The first is a postprocessing heuristic that converts each route in the final solution into its TSP-optimal counterpart. While computationally cheap, this heuristic approach can yield sizeable optimality gaps. The second extension enforces TSP-optimality directly in the pricing algorithm. This makes the pricing problem substantially harder and leads to a near-enumeration of feasible routes for larger instances, making it intractable in practice. \citet{ljubic2026fairvehicle} propose an exact bilevel optimization approach for the F-CVRP-TSP that extends to nonlinear fairness measures, such as the variance, but report optimality gaps around 80--90\% after one hour on instances with $n =25$ customers.

In this paper, we propose a branch-price-and-cut framework that enforces TSP-optimality efficiently through a new family of cutting planes, which we call \emph{TSP-optimality cuts}. These cuts, in the spirit of the path inequalities of \citet{ascheuer2000polyhedral} and \citet{kallehauge2007path}, forbid TSP-dominated arc sequences from appearing in selected routes. To the best of our knowledge, we are the first to use path inequalities to enforce TSP-optimality in a routing application, as compared to using them to enforce other route constraints. These cuts are \emph{robust} in the sense that they do not increase the complexity of the pricing problem. We propose an efficient breadth-first search separation procedure, and strengthen the cuts through a lifting procedure. Computational experiments on benchmark instances show that our method solves almost all instances to proven optimality, with an average optimality gap of only $0.27\%$ on the most difficult configurations with $n=25$. While we focus on the range objective, our cuts are compatible with any arc- or route-based formulation in which TSP-optimality must be enforced.

The remainder of this paper is structured as follows. We formally define the F-CVRP(-TSP) in Section~\ref{sec:problem}, and present the branch-price-and-cut framework of \CitetVR{rossum2025efficient} in Section~\ref{sec:branchPrice}. We introduce the TSP-optimality cuts and a separation and lifting procedure in Section~\ref{sec:cuts}. Finally, we analyze the effectiveness of the cuts through computational experiments in Section~\ref{sec:experiments}.

\section{Problem Description}
\label{sec:problem}

We consider vehicle routing on a complete directed graph $G = (V, A)$, where $V = N \cup \{0\}$ consists of a set of $N$ customers and a depot indexed by $0$. Each customer $i \in N$ has a positive integer demand $d_i \in \mathbb{N}_+$. A fleet of $K$ homogeneous vehicles is available, each with capacity $Q \geq\max_{i\in N}d_i$. The travel distance from node $i$ to node $j$, with $i, j \in N \cup \{0\}$, is given by $d_{ij}$. The total routing distance of all vehicles may not exceed a given budget $L$.

A \emph{route} is an elementary cycle that starts and ends at the depot and visits one or more customers, such that the total demand of visited customers does not exceed the vehicle capacity $Q$. The \emph{length} of a route equals the sum of arc distances along the route. A feasible solution consists of $K$ routes that cover every customer exactly once and whose total length does not exceed $L$.

The \emph{range} of a solution is defined as the difference between the length of the longest route and the length of the shortest route. The goal of the F-CVRP is to find a feasible solution that minimizes the range.

A route is \emph{TSP-optimal} if no other route visiting the same customers, starting and ending at the depot, has strictly smaller length, and \emph{TSP-violating} otherwise. F-CVRP-TSP is the restriction of F-CVRP in which all routes in a feasible solution are required to be TSP-optimal.

\section{Branch-Price-and-Cut Framework}
\label{sec:branchPrice}

We present a branch-price-and-cut framework for F-CVRP-TSP that, up to some minor improvements and the separation of rounded capacity inequalities, mirrors the approach of \CitetVR{rossum2025efficient}. The method operates in two stages. First, it solves F-CVRP using branch-price-and-cut. Second, a postprocessing step converts the obtained routes into a solution to F-CVRP-TSP using the Held-Karp dynamic program \citep{held1962dynamic}. This heuristic serves both as a benchmark for our exact approach and as the algorithmic foundation on which our method is built. In the remainder of this section, we present the mathematical formulation, describe the column generation procedure and pricing algorithm, outline the branching and cutting plane strategies, and detail the postprocessing step.

\subsection{Mathematical Formulation}
\label{subsec:formulation}

Let $R$ denote the set of all feasible routes. For each route $r \in R$, let $l_r$ denote its length, let $a_{ir} \in \{0,1\}$ indicate whether customer $i \in N$ is visited by route $r$, and let $b_{ir} \in \{0,1\}$ indicate whether customer $i$ is the \emph{last} customer visited before returning to the depot on route $r$. We introduce continuous variables $\eta$ and $\gamma$ representing the length of the longest and shortest route in the solution, respectively, and binary variables $x_r \in \{0,1\}$ indicating whether route $r$ is selected. The F-CVRP can then be formulated as
\begin{subequations}
\begin{align}
\min \quad & \eta - \gamma \label{eq:objective} \\
\text{s.t.} \quad  & \sum_{r \in R} a_{ir} x_{r}  = 1 && \forall i \in N \label{eq:partition} \\
& \sum_{r \in R} l_r x_{r} \leq L \label{eq:budget} \\
& \sum_{r \in R} x_{r} = K \label{eq:vehicles} \\
& \sum_{r \in R} l_r b_{ir} x_r \leq \eta && \forall i \in N \label{eq:fcvrp_customer_max} \\
& M \bigl(1 - \sum_{r \in R} b_{ir} x_r \bigr) \nonumber \\
& + \sum_{r \in R} l_r b_{ir} x_r \geq \gamma && \forall i \in N \label{eq:fcvrp_customer_min} \\
& x_{r} \in \{0, 1\} && \forall r \in R. \label{eq:domain}
\end{align}
\label{eq:formulation}
\end{subequations}
The objective~\eqref{eq:objective} is to minimize the range. Constraint~\eqref{eq:partition} ensures every customer is visited exactly once. Constraint~\eqref{eq:budget} enforces the total budget, and~\eqref{eq:vehicles} requires exactly $K$ routes to be used. Constraints~\eqref{eq:fcvrp_customer_max} and~\eqref{eq:fcvrp_customer_min} link route lengths to the upper and lower bound variables $\eta$ and $\gamma$, respectively. The big-$M$ term in~\eqref{eq:fcvrp_customer_min} deactivates the constraint for customers that are not the last on any route. This formulation typically yields tighter linear-programming (LP) bounds than a vehicle-indexed formulation.

Formulation~\eqref{eq:formulation} can be converted into a valid F-CVRP-TSP formulation by restricting $R$ to $R_{TSP} \subseteq R$, the subset of TSP-optimal routes. It is computationally hard to directly enforce this restriction in the pricing problem: \CitetVR{rossum2025efficient} show that this quickly becomes an intractable near-enumeration of feasible routes, and instead propose a postprocessing heuristic that converts the optimal F-CVRP solution into a feasible F-CVRP-TSP solution. In Section~\ref{sec:cuts}, we pursue a different direction: instead of modifying the pricing problem, we add cutting planes to the master problem that eliminate TSP-violating routes, yielding an exact method that preserves the tractability of the pricing algorithm.

\subsection{Column Generation}
\label{subsec:colgen}

Since the number of feasible routes $|R|$ is exponential in the number of customers, the LP relaxation of~\eqref{eq:formulation} is solved using column generation. We initialize a restricted master problem (RMP) with a subset of routes and iteratively add routes with negative reduced costs by solving a \emph{pricing problem}. The algorithm terminates once no route with negative reduced cost exists, certifying LP optimality.

Let $\mu_i \in \mathbb{R}$, $\lambda \geq 0$, $\sigma \in \mathbb{R}$, $\alpha_i \geq 0$, and $\beta_i \geq 0$ denote the dual variables associated with constraints~\eqref{eq:partition},~\eqref{eq:budget},~\eqref{eq:vehicles},~\eqref{eq:fcvrp_customer_max}, and~\eqref{eq:fcvrp_customer_min}, respectively. Fixing the last customer $i \in N$, the reduced cost of a route $r$ with last customer $i$ reads as
\begin{align}
    \tilde{c}_r = (\lambda + \alpha_i - \beta_i)\, l_r + M \beta_i - \sum_{j \in N} a_{jr} \mu_j - \sigma,
    \label{eq:reduced_cost}
\end{align}
where we use the fact that $b_{ir} = 1$ and $b_{jr} = 0$ for all $j \neq i$. Since $l_r$ is a sum of arc distances and $\sum_{j \in N} a_{jr} \mu_j$ decomposes over visited customers, the reduced cost~\eqref{eq:reduced_cost} decomposes over the arcs of the route. Consequently, by initializing one pricing problem per last customer $i \in N$, finding a route with negative reduced cost reduces to an elementary resource-constrained shortest path problem.

The pricing problem is solved on a load-expanded graph, in which each customer node is duplicated for every possible accumulated demand value. The resulting graph is directed and acyclic, which makes pricing considerably more efficient than on the original cyclic graph $G$. The pricing problem is solved using a bidirectional labeling algorithm combined with $ng$-route relaxation \citep{costa2019exact}.

\subsection{Branching and Cutting}
\label{subsec:branching}

When the column generation algorithm terminates with a fractional solution, we apply branching to recover integrality. We use a range branching strategy that branches on the values of $\eta$ and $\gamma$, imposing lower and upper bounds on route lengths and forbidding routes that violate these bounds in the pricing problem. Completion bounds on route distances, precomputed once, are used to prune infeasible partial routes during pricing. Additional branching is performed on last-customer assignments, by forbidding the corresponding pricing problem, and on individual arcs, by excluding their use in the pricing graph.

To strengthen the LP relaxation, we separate cutting planes at each node of the branch-and-bound tree. We add rounded capacity inequalities, separated using the heuristic separation routine of \citet{lysgaard2004new}. To obtain improved upper bounds early in the search, we invoke a primal heuristic every ten branch-and-bound nodes. A restricted master heuristic (RMH) solves the integer program~\eqref{eq:formulation} restricted to the set of columns currently present in the RMP, subject to a time limit of five seconds.

\subsection{Postprocessing Heuristic}
\label{subsec:postprocess}

Since the solutions obtained using the branch-price-and-cut framework are not necessarily TSP-optimal, \CitetVR{rossum2025efficient} propose a postprocessing step where they use the well-known Held-Karp dynamic program \citep{held1962dynamic} to convert all TSP-violating routes into their TSP-optimal counterparts. This yields a feasible solution to the F-CVRP-TSP. The lower bound returned by the branch-price-and-cut algorithm for the F-CVRP remains valid for the F-CVRP-TSP, since the latter problem is a restriction of the former. 

\section{Enforcing TSP-Optimality}
\label{sec:cuts}

Our core idea is to augment the branch-price-and-cut framework of Section~\ref{sec:branchPrice} with a new family of cutting planes that explicitly forbid TSP-violating routes. The cuts eliminate the need for a postprocessing step and yield a fully exact solution approach to the F-CVRP-TSP. We introduce the TSP-optimality cuts in Section~\ref{subsec:cuts}, and present an efficient separation algorithm in Section~\ref{subsec:separation}. We describe a lifting procedure in Section~\ref{subsec:lifting}, and discuss a minor modification to the primal heuristic in Section~\ref{subsec:primal}.

\subsection{TSP-Optimality Cuts}
\label{subsec:cuts}

We call a path $P = (v_1, \ldots, v_p)$ in $G$ \emph{TSP-violating} if there exists a path visiting the same nodes, starting at $v_1$ and ending at $v_p$, with strictly smaller total distance. Without loss of validity, we only consider paths $P$ that do not contain the depot as an intermediate node, i.e., for which
\begin{equation}
    v_2, \ldots, v_{p-1} \neq 0. \label{aspt:depot}
\end{equation}
Let $A(P) = \{(v_1, v_2), \ldots, (v_{p-1}, v_p)\} \subset A$ denote the set of arcs in $P$. If $P$ is TSP-violating, then no TSP-optimal route can contain all arcs in $A(P)$ simultaneously. This motivates the following cutting plane:
\begin{align}
    \sum_{a \in A(P)} x_a \leq |A(P)| - 1, 
    \label{eq:tsp_cut}
\end{align}
where, with a slight abuse of notation, we write $x_a = \sum_{r \in R: a \in r} x_r$ to denote the total flow over arc $a \in A$. Constraint~\eqref{eq:tsp_cut} stipulates that the arcs of a TSP-violating path $P$ cannot all appear together in selected routes. Clearly, any integer solution to~\eqref{eq:formulation} satisfying all such constraints consists exclusively of TSP-optimal routes.

These cuts are closely related to the path inequalities of \citet{ascheuer2000polyhedral}, originally introduced to strengthen the LP relaxation of arc-based formulations for the asymmetric TSP with time windows and later extended by \citet{kallehauge2007path} to the capacitated vehicle routing problem with time windows (CVRPTW). In those settings, the inequalities eliminate arc sequences rendered infeasible by time window or capacity constraints. We adopt the same principle, applying it instead to exclude arc sequences that are suboptimal from a TSP perspective.

A key property of the TSP-optimality cuts is that they are \emph{robust}: the dual variable of each cut~\eqref{eq:tsp_cut} decomposes naturally over arcs in the pricing graph. Consequently, adding TSP-optimality cuts to the master problem does not alter the structure of the pricing problem, and the labeling algorithm described in Section~\ref{subsec:colgen} can be directly applied.

\subsection{Separation}
\label{subsec:separation}

Since there are exponentially many TSP-optimality cuts, we separate them dynamically. For each arc $a \in A$, we compute the arc flow $x_{a}$ in the current fractional solution. We then identify violated cuts using a breadth-first search (BFS) over the graph induced by these flows.

The separation procedure works as follows. Starting from a source node, we extend a partial path $P = (v_1, \ldots, v_p)$ by appending a successor node $v_{p+1}$ as long as the total flow along $A(P)$ exceeds $|A(P)| - 1$, i.e., whenever the TSP-optimality cut \eqref{eq:tsp_cut} corresponding to $P$ could be violated. This condition already eliminates the vast majority of candidate paths. For each partial path of length $|P| \geq 3$, we invoke the Held-Karp dynamic program \citep{held1962dynamic} to compute the length of the shortest path visiting the same nodes as $P$, starting at $v_1$ and ending at $v_p$. If this shortest path is strictly shorter than $P$, then $P$ is TSP-violating and the corresponding cut~\eqref{eq:tsp_cut} is violated by the current solution; we add it to the RMP.

In principle, the BFS can be initiated from any node in $G$. In practice, we find that restricting separation to paths that start or end at the depot already returns the most violated cuts. The separation procedure is computationally cheap, allowing us to separate TSP-optimality cuts at both fractional and integer solutions, as is standard practice in branch-price-and-cut algorithms for vehicle routing \citep{costa2019exact}.

\subsection{Lifting}
\label{subsec:lifting}

Using a combinatorial lifting argument, \citet{mak2001asymmetric} proposes a strengthened version of path inequalities for the asymmetric traveling salesman problem in an arc-based formulation. \citet{kallehauge2007path} extend the idea to an arc-based formulation of CVRPTW. Similar techniques can be adapted in our setting to lift the TSP-optimality cuts \eqref{eq:tsp_cut}. The core idea is to identify additional arcs that, when combined with arcs of the path, would also lead to an infeasible partial route, e.g., by violating elementarity, TSP-optimality, or the vehicle capacity constraint. These arcs can then be appended to the left-hand side of the cut without changing its validity, yielding a stronger inequality. Throughout this section, we use the shorthand $x(A) = \sum_{a \in A} x_a$ for any arc set $A$.

Consider a TSP-violating path $P=(v_1,\ldots,v_p)$. For any partial path $(v_1, \ldots, v_h)$ of $P$, we define the following sets:
\begin{align*}
  &\Delta^{E,+}_{h}(P) &=~&  \big\{(v_h, v_j): 1 \leq j \leq h-1, v_j \neq 0 \big\}, \\[4pt]
  &\Delta^{TSP,+}_{h}(P) &=~& \big\{(v_h,j) : j\notin \{v_1,\ldots,v_{h+1}\} \text{ s.t.} \\
  &&& (v_1,\ldots,v_h,j)\text{ is TSP-violating} \big\}, \\[4pt]
  &\Delta^{Q,+}_{h}(P) &=~& \big\{(v_h,j) : j\notin \{v_1,\ldots,v_{h+1}\} \text{ s.t.} \\
  &&& \textstyle\sum_{i=1}^{h} d_{v_i} + d_j > Q \big\}.
\end{align*}
Observe that they correspond to the forward arcs that, when added to the partial path $(v_1, \ldots, v_h)$, render the resulting path infeasible due to a violation of elementarity, TSP-optimality, or the capacity constraint. Next, define $\Delta^{+}_{h}(P) = \Delta^{E, +}_{h}(P) \cup \Delta^{TSP,+}_{h}(P) \cup \Delta^{Q, +}_{h}(P)$. Then, the partial path $(v_1, \ldots, v_h)$ becomes infeasible when extended along any arc in $\Delta^{+}_{h}(P)$. The forward lifted TSP-optimality cuts read as
\begin{equation}
x(A(P))+ x\!\big(\bigcup_{h=1}^{p-1} \Delta^{+}_{h}(P)\big)\leq |A(P)|-1. \label{eq:lifted_cut}
\end{equation}
We provide a stand-alone proof of their validity, largely following the reasoning of \citet{kallehauge2007path}.
\begin{prop}
For a given TSP-violating path $P$, the forward-lifted cut~\eqref{eq:lifted_cut} is valid for F-CVRP-TSP.
\end{prop}
\begin{proof}
Assume to the contrary that there exists a feasible solution $x$ violating~\eqref{eq:lifted_cut}, i.e., for which the left-hand side equals at least $|A(P)|$. We prove by induction that $x_{v_h, v_{h+1}} = 1$ for $h = 1, \ldots, p-1$, which implies $x(A(P)) = |A(P)|$, contradicting TSP-optimality of $x$.

Observe that the cut can be rewritten as
\begin{equation}
\sum_{h=1}^{p-1} \big( x_{v_h, v_{h+1}} + x(\Delta^{+}_{h}(P)) \big) \leq |A(P)| - 1. \label{eq:lifted_cut_rewritten}
\end{equation}
We note three facts. First, by \eqref{aspt:depot}, the node $v_h$ is a customer node for all $h = 2, \ldots, p-1$. By the degree constraint, it follows that $x_{v_h, v_{h+1}} + x(\Delta^{+}_{h}(P)) \leq 1$ for all such $h$. Second, $\Delta^{E,+}_{1}(P) = \Delta^{TSP,+}_{1}(P) = \emptyset$, since the partial path $(v_1)$ consists of a single node. Third, $x_a=0$ for $a\in \Delta^{Q,+}_{1}(P)$ by feasibility of $x$ and the definition of $\Delta^{Q,+}_{1}(P)$.

Hence, each of the $p-1$ terms in~\eqref{eq:lifted_cut_rewritten} contributes at most $1$, and the first term contributes at most $x_{v_1, v_2}$. For the left-hand side to reach $|A(P)| = p - 1$, it must hold that $x_{v_1, v_2} = 1$ and $x_{v_h, v_{h+1}} + x(\Delta^{+}_{h}(P)) = 1$ for all $h = 2, \ldots, p-1$.

We now apply induction on $h$. The base case $h=1$ requires $x_{v_1, v_2} = 1$, as already argued above.

For the induction step, suppose $x_{v_h, v_{h+1}} = 1$ for all $h = 1, \ldots, k-1$ with $2 \leq k \leq p-1$. We aim to show that $x_{v_k, v_{k+1}} = 1$. As the cut is violated, we have $x_{v_k, v_{k+1}} + x(\Delta^{+}_{k}(P)) = 1$. Suppose $x_{v_k, v_{k+1}} = 0$. Then $x(\Delta^{+}_{k}(P)) = 1$, meaning there exists an arc $(v_k, j) \in \Delta^{+}_{k}(P)$ with $x_{v_k, j} = 1$. By the induction hypothesis, the partial path $(v_1, \ldots, v_k)$ is traversed by $x$. However, by construction of $\Delta^{+}_{k}(P)$, extending this path with arc $(v_k, j)$ violates elementarity, TSP-optimality, or the capacity constraint. In each case, this contradicts feasibility of $x$. Hence, $x_{v_k, v_{k+1}} = 1$.

By induction, we conclude that path $P$ is fully traversed, contradicting TSP-optimality of $x$.
\end{proof}

\begin{figure}
\centering
\begin{tikzpicture}[scale=1,
    midarrow/.style={
        decoration={markings, mark=at position 0.5 with {\arrow{stealth}}},
        postaction={decorate}
    }
]
\begin{axis}[
    axis lines = none,
    disabledatascaling,
    xmin=130, xmax=900, ymin=-50, ymax=950,
    ]

\tikzset{every node/.style={draw=black, fill=black!15, shape=circle, scale=0.5}}

\node[shape=rectangle, scale=2, fill=white] at (500, 80) (0) { };

\node at (220, 700) (v1) { };
\node at (780, 200) (v2) { };
\node at (780, 700) (v3) { };
\node at (500, 850) (v4) { };

\draw[black, thick, midarrow] (0) -- (v1);
\draw[black, thick, midarrow] (v1) -- (v2);
\draw[black, thick, midarrow] (v2) -- (v3);

\draw[thick, dashed, midarrow, bend left=20]  (v2) to (v1);
\draw[thick, dashed, midarrow]                (v3) -- (v1);
\draw[thick, dashed, midarrow, bend left=20]  (v3) to (v2);

\draw[black, thick, dotted, midarrow] (v2) -- (v4);

\node[draw=none, fill=none, scale=2, font=\small] at (160, 700) {$v_1$};
\node[draw=none, fill=none, scale=2, font=\small] at (840, 200) {$v_2$};
\node[draw=none, fill=none, scale=2, font=\small] at (840, 700) {$v_3$};
\node[draw=none, fill=none, scale=2, font=\small] at (500, 910) {$v_4$};

\end{axis}
\end{tikzpicture}
\caption{Forward lifting of a TSP-optimality cut. Solid arrows show the TSP-violating path $P = (0, v_1, v_2, v_3)$, which is longer than the path $(0, v_2, v_1, v_3)$. Dashed arrows are arcs lifted via route elementarity: they cannot coexist with all arcs of $P$ in an elementary route. The dotted arrow $(v_2, v_4)$ is lifted via TSP-optimality: the path $(0, v_1, v_2, v_4)$ is longer than $(0, v_2, v_1, v_4)$.}
\label{fig:lifting}
\end{figure}
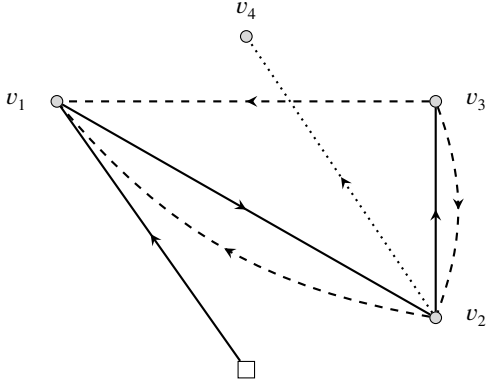

Figure~\ref{fig:lifting} illustrates the forward lifting procedure. Starting from the TSP-violating path $P=(0, v_1, v_2, v_3)$, the associated TSP-optimality cut~\eqref{eq:tsp_cut} reads
\begin{equation*}
x_{0, v_1} + x_{v_1, v_2} + x_{v_2, v_3} \leq 2.
\end{equation*}
Forward lifting based on elementarity and TSP-optimality yields the significantly stronger cut
\begin{equation*}
x_{0, v_1} + x_{v_1, v_2} + x_{v_2, v_1} + x_{v_2, v_3} + x_{v_2, v_4} + x_{v_3, v_1} + x_{v_3, v_2} \leq 2.
\end{equation*}

\paragraph{Backward lifting.}
We can apply an analogous lifting procedure by traversing path $P$ in the reverse direction. For each partial suffix $(v_h, \ldots, v_p)$, define the backward lifting sets
\begin{align*}
  &\Delta^{E,-}_{h}(P) &=~& \big\{(v_j, v_h) : h+1 \leq j \leq p, v_j \neq 0 \big\}, \\[4pt]
  &\Delta^{TSP,-}_{h}(P) &=~& \big\{(j, v_h) : j\notin \{v_{h-1},\ldots,v_p\} \text{ s.t.} \\
  &&& (j, v_h, \ldots, v_p)\text{ is TSP-violating} \big\}, \\[4pt]
  &\Delta^{Q,-}_{h}(P) &=~& \big\{(j, v_h) : j\notin \{v_{h-1},\ldots,v_p\} \text{ s.t.} \\
  &&& d_j + \textstyle\sum_{i=h}^{p} d_{v_i} > Q \big\}.
\end{align*}
Setting $\Delta^{-}_{h}(P) = \Delta^{E,-}_{h}(P) \cup \Delta^{TSP,-}_{h}(P) \cup \Delta^{Q,-}_{h}(P)$, the backward-lifted TSP-optimality cut reads
\begin{equation}
x(A(P)) + x\!\left(\bigcup_{h=2}^{p} \Delta^{-}_{h}(P)\right) \leq |A(P)| - 1. \label{eq:backward_lifted_cut}
\end{equation}
Its validity follows by an analogous induction argument traversing the path from $v_p$ back to $v_1$: if the cut is violated, then by the degree constraints and the definition of the backward lifting sets, the arcs of $P$ must all be traversed, contradicting TSP-optimality. In our implementation, we apply both forward and backward lifting to each cut.

\subsection{Primal Heuristic}
\label{subsec:primal}

A subtlety arises when applying the RMH in the context of the F-CVRP-TSP. Since the restricted master problem may contain TSP-violating routes, an integral solution returned by the RMH is not guaranteed to be feasible for the F-CVRP-TSP, even if all TSP-optimality cuts generated so far are satisfied. To address this, we convert all routes in the restricted master problem into their TSP-optimal counterparts before invoking the RMH. This ensures that any integral solution produced by the heuristic is feasible for the F-CVRP-TSP.

\section{Computational Experiments}
\label{sec:experiments}

We evaluate our method on the same instances as \CitetVR{rossum2025efficient} and \citet{ljubic2026fairvehicle}, containing $K=5$ vehicles, $n \in \{15, 20, 25\}$ customers, and using a route budget $L$ equal to 101, 105, or 110\% of the minimum total routing distance. We consider 20 different instances per configuration, yielding 180 instances in total. All algorithms are implemented in \texttt{Java}, our code is available upon request. All (integer) linear programs are solved using \texttt{CPLEX 22.1.0}. Experiments are run on computing cluster nodes equipped with 16\,GB of RAM and an AMD Rome 7H12 processor, with a time limit of one hour per instance.


\begin{table*}[htbp!]
  \centering
  \caption{%
    Average computational results per configuration.}
  \label{tab:results_tsp}
  \setlength{\tabcolsep}{4pt}
  \begin{tabular}{rrlrrrrrrr}
    \toprule
    $n$ & $L$ (\%) & Algorithm
        & LB & UB & Gap (\%) & Time (s) & Nodes & Cuts & Optimal \\
    \midrule
    \multirow{9}{*}{15} & \multirow{3}{*}{101} & Postprocessing & 1,765.25 & 1,765.40 & 0.01 & 4 & 250 & 37 & 19/20 \\
     &  & Cuts & 1,765.40 & 1,765.40 & 0 & 4 & 227 & 55 & 20/20 \\
     &  & Cuts + Lifting & 1,765.40 & 1,765.40 & 0 & 4 & 222 & 60 & 20/20 \\
    \cmidrule{2-10}
     & \multirow{3}{*}{105} & Postprocessing & 1,417.60 & 1,427.05 & 1.16 & 7 & 304 & 52 & 15/20 \\
     &  & Cuts & 1,425.25 & 1,425.25 & 0 & 7 & 299 & 88 & 20/20 \\
     &  & Cuts + Lifting & 1,425.25 & 1,425.25 & 0 & 7 & 277 & 97 & 20/20 \\
    \cmidrule{2-10}
     & \multirow{3}{*}{110} & Postprocessing & 965.65 & 1,031.25 & 7.73 & 6 & 193 & 48 & 6/20 \\
     &  & Cuts & 999.75 & 999.75 & 0 & 16 & 406 & 106 & 20/20 \\
     &  & Cuts + Lifting & 999.75 & 999.75 & 0 & 11 & 276 & 105 & 20/20 \\
    \midrule
    \multirow{9}{*}{20} & \multirow{3}{*}{101} & Postprocessing & 1,506.50 & 1,508.90 & 0.36 & 25 & 275 & 52 & 17/20 \\
     &  & Cuts & 1,508.90 & 1,508.90 & 0 & 21 & 225 & 92 & 20/20 \\
     &  & Cuts + Lifting & 1,508.90 & 1,508.90 & 0 & 19 & 187 & 102 & 20/20 \\
    \cmidrule{2-10}
     & \multirow{3}{*}{105} & Postprocessing & 1,083.75 & 1,120.10 & 4.09 & 39 & 364 & 59 & 8/20 \\
     &  & Cuts & 1,102.30 & 1,102.30 & 0 & 54 & 471 & 160 & 20/20 \\
     &  & Cuts + Lifting & 1,102.30 & 1,102.30 & 0 & 56 & 386 & 170 & 20/20 \\
    \cmidrule{2-10}
     & \multirow{3}{*}{110} & Postprocessing & 768.65 & 842.70 & 13.63 & 41 & 360 & 58 & 3/20 \\
     &  & Cuts & 795.55 & 795.55 & 0 & 247 & 1,660 & 263 & 20/20 \\
     &  & Cuts + Lifting & 795.55 & 795.55 & 0 & 159 & 968 & 218 & 20/20 \\
    \midrule
    \multirow{9}{*}{25} & \multirow{3}{*}{101} & Postprocessing & 1,594.40 & 1,600.65 & 0.51 & 356 & 772 & 77 & 13/20 \\
     &  & Cuts & 1,600.25 & 1,600.25 & 0 & 333 & 713 & 166 & 20/20 \\
     &  & Cuts + Lifting & 1,600.25 & 1,600.25 & 0 & 299 & 627 & 193 & 20/20 \\
    \cmidrule{2-10}
     & \multirow{3}{*}{105} & Postprocessing & 1,151.85 & 1,184.95 & 3.96 & 271 & 698 & 81 & 6/20 \\
     &  & Cuts & 1,174.65 & 1,174.65 & 0 & 760 & 2,049 & 305 & 20/20 \\
     &  & Cuts + Lifting & 1,174.65 & 1,174.65 & 0 & 468 & 1,189 & 258 & 20/20 \\
    \cmidrule{2-10}
     & \multirow{3}{*}{110} & Postprocessing & 751.37 & 867.05 & 16.50 & 504 & 1,113 & 71 & 3/20 \\
     &  & Cuts & 771.45 & 774.70 & 0.67 & 1,616 & 3,838 & 500 & 15/20 \\
     &  & Cuts + Lifting & 773.10 & 774.80 & 0.27 & 1,310 & 2,668 & 377 & 16/20 \\
    \bottomrule
  \end{tabular}
\end{table*}

Table~\ref{tab:results_tsp} reports the lower bound (LB), upper bound (UB), optimality gap, computing time, number of branch-and-bound nodes, number of cuts added, and number of instances solved to optimality, averaged per configuration. Note that the number of cuts measures both rounded capacity inequalities and TSP-optimality cuts. We compare three methods, all building on the branch-price-and-cut framework of Section~\ref{sec:branchPrice}: the postprocessing heuristic of \CitetVR{rossum2025efficient} that converts routes to TSP-optimal ones after solving the F-CVRP (`Postprocessing'), the same framework augmented with TSP-optimality cuts (`Cuts'), and the same framework with lifted TSP-optimality cuts (`Cuts + Lifting').

The postprocessing heuristic performs well on small instances with tight budgets, as previously observed \citep{rossum2025efficient}. For $n = 15$ and $L = 101\%$, it achieves an average gap of $0.01\%$ and solves 19 out of 20 instances to optimality within four seconds. Its performance deteriorates as the budget increases: for $n = 25$ and $L = 110\%$, the average gap rises to $16.50\%$ and only 3 out of 20 instances are solved to optimality, suggesting that more lenient budgets allow greater TSP-violations in the underlying solution.

The exact algorithm with regular TSP-optimality cuts obtains a near-zero optimality gap on all but the hardest configuration ($n = 25$, $L = 110\%$), where it achieves an average gap of $0.67\%$ and solves 15 out of 20 instances within the time limit. Observe that these optimality gaps are two orders of magnitude below those of the bilevel optimization approach of \citet{ljubic2026fairvehicle}. Compared to the postprocessing heuristic, solution quality improves substantially in the higher-budget settings, though this comes at increased computation time. 

Lifting the TSP-optimality cuts is generally beneficial, reducing the number of branch-and-bound nodes, cuts added, and computation time relative to the non-lifted baseline across many configurations. For the largest instances ($n = 25$, $L = 110\%$), the average gap decreases from $0.67\%$ to $0.27\%$, the node count from $3{,}838$ to $2{,}668$, the number of cuts from $500$ to $377$, and computation time from $1{,}616$ to $1{,}310$ seconds, while the number of instances solved to optimality increases from 15 to 16 out of 20. The reduction in cuts added is notable, as lifting produces fewer but stronger inequalities. We conclude that the proposed TSP-optimality cuts, and their lifted variants in particular, form an effective way of enforcing TSP-optimality in fair vehicle routing.

\bibliographystyle{cas-model2-names}
\bibliography{references}

@article{rossum2025efficient,
  title={Efficient branching rules for optimizing range and order-based objective functions},
  author={van Rossum, Bart and Chen, Rui and Lodi, Andrea},
  journal={Mathematical Programming},
  pages={1--34},
  year={2025},
  publisher={Springer}
}

@article{bruinink2025beyond,
title = {Beyond efficiency: Exploring the cost effects of prioritizing driver satisfaction in vehicle routing},
journal = {Operations Research Letters},
volume = {63},
pages = {107344},
year = {2025},
issn = {0167-6377},
author = {Tom Bruinink and Lotte Berghman and Twan Dollevoet},
}

@article{matl2018workload,
  title={Workload equity in vehicle routing problems: A survey and analysis},
  author={Matl, Piotr and Hartl, Richard F and Vidal, Thibaut},
  journal={Transportation Science},
  volume={52},
  number={2},
  pages={239--260},
  year={2018},
  publisher={INFORMS}
}

@article{lysgaard2004new,
  title={A new branch-and-cut algorithm for the capacitated vehicle routing problem},
  author={Lysgaard, Jens and Letchford, Adam N and Eglese, Richard W},
  journal={Mathematical programming},
  volume={100},
  number={2},
  pages={423--445},
  year={2004},
  publisher={Springer}
}

@article{ascheuer2000polyhedral,
  title={A polyhedral study of the asymmetric traveling salesman problem with time windows},
  author={Ascheuer, Norbert and Fischetti, Matteo and Gr{\"o}tschel, Martin},
  journal={Networks},
  volume={36},
  number={2},
  pages={69--79},
  year={2000},
  publisher={Wiley Online Library}
}

@misc{ljubic2026fairvehicle,
  author    = {Ljubi{\'c}, Ivana and Puerto, Justo and Torrejon, Alberto},
  title     = {Fair Vehicle Routing via Bilevel Optimization},
  year      = {2026},
  month     = {March},
  url       = {https://optimization-online.org/2026/03/fair-vehicle-routing-via-bilevel-optimization/},
  note      = {{P}reprint, submitted March 5, 2026}
}

@article{kallehauge2007path,
author = {Kallehauge, Brian and Boland, Natashia and Madsen, Oli B.G.},
title = {Path inequalities for the vehicle routing problem with time windows},
journal = {Networks},
volume = {49},
number = {4},
pages = {273-293},
year = {2007}
}

@article{agius2026solution,
  title={Solution algorithms for a vehicle routing problem with route-cost equity constraints},
  author={Agius, Maxime and Absi, Nabil and Feillet, Dominique and Garaix, Thierry},
  journal={Annals of Operations Research},
  pages={1--41},
  year={2026},
  publisher={Springer}
}

@inproceedings{van2024new,
  title={A new branching rule for range minimization problems},
  author={van Rossum, Bart and Chen, Rui and Lodi, Andrea},
  booktitle={International Conference on Integer Programming and Combinatorial Optimization},
  pages={433--445},
  year={2024},
  organization={Springer}
}

@article{costa2019exact,
  title={Exact branch-price-and-cut algorithms for vehicle routing},
  author={Costa, Luciano and Contardo, Claudio and Desaulniers, Guy},
  journal={Transportation Science},
  volume={53},
  number={4},
  pages={946--985},
  year={2019},
  publisher={INFORMS}
}

@article{held1962dynamic,
  title={A dynamic programming approach to sequencing problems},
  author={Held, Michael and Karp, Richard M},
  journal={Journal of the Society for Industrial and Applied Mathematics},
  volume={10},
  number={1},
  pages={196--210},
  year={1962},
  publisher={SIAM}
}

@phdthesis{mak2001asymmetric,
  title={On the Asymmetric Travelling Salesman Problem with Replenishment Arcs},
  author={Mak, Vicky H},
  year={2001},
  school={University of Melbourne, Department of Mathematics and Statistics}
}

\end{document}